\documentclass[11pt,english]{article}
\usepackage[T1]{fontenc}
\usepackage[latin9]{inputenc}
\usepackage{array}
\usepackage{amsmath}
\usepackage{graphicx}
\usepackage{amssymb}
\usepackage{array}
\usepackage{float}
\usepackage{color}

\newtheorem{theorem}{Theorem}
\newtheorem{lemma}{Lemma}

\newtheorem{definition}{Definition}

\newtheorem{proposition}{Proposition}
\newtheorem{remark}{Remark}

\newenvironment{proof}{\begin{trivlist}
\item[]{\bf Proof. }}{\hspace*{\fill}$\rule{.3\baselineskip}{.35\baselineskip}$
\end{trivlist}}
\newenvironment{proof1}{
    \noindent {\bf Proof }}{\hfill$\Box$}

\newcommand{\Dscr}{\mathcal{D}}
\newcommand{\Xscr}{\mathcal{X}}

\newcommand{\Z}{\mathbb{Z}}

\newcommand{\eps}{\epsilon}
\newcommand{\norm}[1]{\left\|#1\right\|}
\newcommand{\tond}[1]{{\left(#1\right)}}
\newcommand{\quadr}[1]{{\left[#1\right]}}

\newcommand{\Ham}[2]{H^{(#1)}_{#2}}
\newcommand{\Poi}[2]{\{#1,#2\}}
\newcommand{\derp}[2]{{\frac{\partial #1}{\partial #2}}}
\newcommand{\ncamp}[1]{\Big\|#1\Big\|^\oplus}
\newcommand{\Zscr}{\mathcal{Z}}
\newcommand{\Pscr}{\mathcal{P}}
\newcommand{\K}{\mathcal K}
\newcommand{\Keff}{\K_{\rm eff}}
\newcommand{\Kres}{\K_{\rm res}}
\newcommand{\Res}{{\rm Res}}
\newcommand{\Cst}{\mathsf{C}}
\newcommand{\Lin}{\mathcal L}

\def\Chi{\Xscr}

\def\diam{\mathop{\rm diam}}

\def\Id{\mathbb I}

\begin{document}

\title{Approximation of small-amplitude weakly coupled oscillators\\
  with discrete nonlinear Schr\"{o}dinger equations}

\author{Dmitry Pelinovsky$^{1,2}$, Tiziano Penati$^{3}$, and Simone Paleari$^{3}$ \\
{\small $^{1}$
    Department of Mathematics, McMaster University, Hamilton, Ontario,
    Canada, L8S 4K1} \\ {\small $^{2}$ Department of Applied Mathematics, Nizhny Novgorod State Technical University,} \\
{\small 24 Minin Street, Nizhny Novgorod, 603950, Russia}\\
{\small $^{3}$ Department of Mathematics
    ``F.Enriques'', Milano University, via Saldini 50, Milano, Italy, {20133}} }

\date{\today}
\maketitle

\begin{abstract}
Small-amplitude weakly coupled oscillators of the Klein--Gordon
lattices are approximated by equations of the discrete nonlinear
Schr\"{o}dinger type.  We show how to justify this approximation by
two methods, which have been very popular in the recent
literature. The first method relies on a priori energy estimates and
multi-scale decompositions. The second method is based on a resonant
normal form theorem.  We show that although the two methods are
different in the implementation, they produce equivalent results as
the end product. We also discuss applications of the discrete nonlinear Schr\"{o}dinger
equation in the context of existence and stability of breathers of the Klein--Gordon lattice.
\end{abstract}

\section{Introduction}
\label{s:intro}

We consider the one-dimensional discrete Klein--Gordon (dKG) equation
with the hard quartic potential in the form
\begin{equation}
\ddot{x}_j + x_j + x_j^3 = \epsilon (x_{j+1}-2x_{j}+x_{j-1}),\quad
j\in\mathbb{Z},\label{KGlattice-initial}
\end{equation}
where $t\in \mathbb{R}$ is the evolution time, $x_j(t)\in \mathbb{R}$
is the horizontal displacement of the $j$-th particle in the
one-dimensional chain, and $\epsilon > 0$ is the coupling constant of
the linear interaction between neighboring particles. The dKG equation
(\ref{KGlattice-initial}) is associated with the conserved-in-time energy
\begin{equation}
H = \frac{1}{2} \sum_{j \in \mathbb{Z}} \dot{x}_j^2 + x_j^2 + \epsilon
(x_{j+1}-x_j)^2 + \frac{1}{4} \sum_{j \in \mathbb{Z}} x_j^4,
\label{energy-Hamiltonian}
\end{equation}
which is also the Hamiltonian function of the dKG equation
(\ref{KGlattice-initial}) written in the canonical variables $\{ x_j,
\dot{x}_j \}_{j \in \mathbb{Z}}$.  The initial-value problem for the dKG
equation (\ref{KGlattice-initial}) is globally well-posed in the sequence
space $\ell^2(\mathbb{Z})$, thanks to the coercivity of the energy $H$ in
(\ref{energy-Hamiltonian}) in $\ell^2(\mathbb{Z})$.

By using a scaling transformation
\begin{equation}
\label{scaling-transformation}
\tilde{x}_j(\tilde{t}) = (1 + 2 \epsilon)^{-1/2} x_j(t), \quad \tilde{t} = (1+2\epsilon)^{1/2} t,
\quad \tilde{\epsilon} = (1+2\epsilon)^{-1} \epsilon,
\end{equation}
and dropping the tilde notations, the dKG equation (\ref{KGlattice-initial}) can be rewritten
without the diagonal terms in the discrete Laplacian operator,
\begin{equation}
\ddot{x}_j + x_j + x_j^3 = \epsilon (x_{j+1}+x_{j-1}),\quad
j\in\mathbb{Z}.\label{KGlattice}
\end{equation}
Note that the values of $\epsilon$ in (\ref{KGlattice}) are now restricted to
the range $\left(0,\frac{1}{2} \right)$, because the map $\epsilon \to
(1+2\epsilon)^{-1} \epsilon$ is a diffeomorphism from $(0,\infty)$ to
$\left(0,\frac{1}{2} \right)$. This restriction does not represent a
limitation if we study solutions of the dKG equation for sufficiently small
values of $\epsilon$.

We consider the Cauchy problem for the dKG equation (\ref{KGlattice}) and we aim
at giving an approximation of its solutions by means of equations of the discrete
nonlinear Schr\"{o}dinger type, up to suitable time scales. This approach can be
useful in general, but it may have additional interest when particular classes of
solutions of the dKG equation (\ref{KGlattice}) are taken into account. In the case
of systems of weakly coupled oscillators, relevant objects are given by time-periodic
and spatially localized solutions called breathers.

Existence and stability of breathers have been studied in the dKG equation in
many recent works. In particular, exploring the limit of weak coupling between
the nonlinear oscillators, existence \cite{MA94} and stability
\cite{Bambusi1,Bambusi2} of the fundamental (single-site) breathers were established
(see also the recent works in \cite{PP14,PP15}).
More complicated multi-breathers were classified from the point of their spectral stability in
the recent works \cite{Archilla,KK09,PelSak}.  Nonlinear stability and instability of multi-site breathers were
recently studied in \cite{CKP15}.

If the oscillators have small amplitudes in addition to being weakly coupled,
the stability of multi-breathers in the dKG equation is related to the
stability of multi-solitons in the discrete nonlinear Schr\"{o}dinger (dNLS)
equation:
\begin{equation}
\label{dnls-introduction}
2 i \dot{a}_j + 3 |a_j|^2 a_j = a_{j+1} + a_{j-1}, \quad j \in \mathbb{Z},
\end{equation}
where $a_j(\epsilon t) \in \mathbb{C}$ is the envelope amplitude for the
linear harmonic $e^{it}$ supported by the linear dKG equation
(\ref{KGlattice}) with $\epsilon = 0$.  The relation between the dKG and dNLS
equations (\ref{KGlattice}) and (\ref{dnls-introduction}) was observed in
\cite{MJKA02} based on numerical simulations and was elaborated in
\cite{PelSak} with perturbation technique.

The present contribution addresses the justification of the dNLS equation
(\ref{dnls-introduction}), and its generalizations, for the weakly coupled small-amplitude oscillators
of the dKG equation (\ref{KGlattice}). In fact, we are going to explore two
alternative but complementary points of view on the justification process,
which enables us to establish rigorous bounds on the error terms over the time
scale, during which the dynamics of the dNLS equation
(\ref{dnls-introduction}) is observed. 

The first method in the justification of the dNLS equation
(\ref{dnls-introduction}) for small-amplitude weakly coupled oscillators of
the dKG equation (\ref{KGlattice}) is based on a priori energy estimates and
elementary continuation arguments. This method was used in the derivation of
the dNLS equation \cite{DumaJames} and the Korteweg--de Vries equation
\cite{BP06,DumaPel,Pego,SW00} in a similar context of the Fermi--Pasta--Ulam
lattice. The energy method is based on the decomposition of the solution into
the leading-order multi-scale approximation and the error term. The error term
is controlled by integrating the dKG equation with a small residual term over relevant time scale.
The energy method is computationally efficient and simple enough for most
practical applications.

The second method is based on the resonant normal form theorem, which
transforms the given Hamiltonian of the dKG equation to a simpler form by
means of near-identity canonical transformations \cite{BambusiPisa,G03}.
The normal form, once it is obtained in the sense of an abstract theorem, 
does not require any additional work for the derivation and justification of both the dNLS
equation and its generalizations, which appear immediately in
the corresponding relevant regimes. Starting from the works
  \cite{GPP12,GPP13}, the normal form approach for the dKG equation was
recently elaborated in \cite{PP14} and applied in \cite{PP15} for a stability result.

We hope that the present discussion of the two equivalent methods can motivate
readers for the choice of a suitable analytical technique in the justification
analysis of similar problems of lattice dynamics. It is our understanding that
the two methods are equivalent with respect to the results (error
estimates, time scales) but have some differences in the way one proves
such results.

Besides justifying the dNLS equation (\ref{dnls-introduction}) on the time scale $\mathcal{O}(\epsilon^{-1})$, 
we also extend the error bounds on the longer time intervals of $\mathcal{O}(|\log(\epsilon)| \epsilon^{-1})$. 
Similar improvements were reported in various other contexts of the justification analysis 
\cite{DumaJames,Holmer}. Within our context, we will show in the end of this paper how to transfer
the known results on existence and stability of multi-solitons in the dNLS
equation (\ref{dnls-introduction}) to the approximate results on existence and
stability of multi-breathers in the dKG equation (\ref{KGlattice}). In
particular, we will address the relevant approximations for the dKG breathers 
on the extended time intervals obtained 
from the analytical results on the asymptotic stability of
dNLS solitons \cite{CT09,KPS09,MP12}, quasi-periodic localized solutions
\cite{Cuccagna,Maeda}, internal modes of dNLS solitons \cite{PelSak11}, and
nonlinear instability of multi-site solitons \cite{KPS15}.

We finish introduction with a review of related results.  Small-amplitude
breathers of the dKG and dNLS equations were approximated with the continuous
nonlinear Schr\"{o}dinger equation in the works \cite{BPP10,BP10,PP12}. An alternative derivation of the
continuous nonlinear Schr\"{o}dinger equation was discussed in the context of
the Fermi--Pasta--Ulam lattice \cite{GM1,GM2,GHM,Schneider}.  In the opposite
direction, derivation and justification of the dNLS equation from a continuous
nonlinear Schr\"{o}dinger equation with a periodic potential were developed in
the works \cite{PSM,PS}.  Finally, justification of the popular variational
approximation for multi-solitons of the dNLS equation in the limit of weak
coupling between the nonlinear oscillators is reported in \cite{Chong}.

The paper is organized as follows. Section 2 reports the justification results
obtained from the energy method and multi-scale expansions. Section 3 reports
the justification results obtained from the normal form theorem.  Section 4
discusses applications of these results for existence and stability of breathers 
in the dKG equation.

\section{Justification of the dNLS equation with the energy method}
\label{s:energy-meth}

In what follows, we consider the limit of weak coupling between the nonlinear
oscillators, where $\epsilon$ is a small positive parameter. We also consider
the small-amplitude oscillations starting with small-amplitude initial
data. Hence, we use the scaling transformation $x_j = \rho^{1/2} \xi_j$, where
$\rho$ is another small positive parameter. Incorporating both small
parameters, we rewrite the dKG equation (\ref{KGlattice}) in the equivalent
form
\begin{equation}
\ddot{\xi}_j + \xi_j + \rho \xi_j^3 = \epsilon (\xi_{j+1}+\xi_{j-1}),\quad
j\in\mathbb{Z}. \label{KG}
\end{equation}

The standard approximation of multi-breathers in the dKG equation
(\ref{KG}) with multi-solitons of the dNLS equation (\ref{dnls-introduction}) corresponds to the
balance $\rho = \epsilon$.  In Sections 2.1--2.3, we
generalize the standard dNLS approximation by assuming that
$\epsilon^2 \ll \rho \leq \epsilon$.  In Section 2.4, we
discuss further generalizations when $\rho$ belongs to the asymptotic
range $\epsilon^3 \ll \rho \leq \epsilon^2$.

\subsection{Preliminary estimates}

To recall the standard dNLS approximation, we define the slowly
varying approximate solution of the dKG equation (\ref{KG}) in the
form
\begin{equation}
X_j(t) = a_{j}(\epsilon t)e^{it} +
\bar{a}_{j}(\epsilon t) e^{-it}.
\label{envelope-approximation}
\end{equation}
Substituting the leading-order solution (\ref{envelope-approximation})
to the dKG equation (\ref{KG}) and removing the resonant terms $e^{\pm
  i t}$ at the leading order of $\mathcal{O}(\epsilon)$, we obtain the
dNLS equation in the form
\begin{equation}
2 i \dot{a}_j + 3 \nu |a_j|^2 a_j = a_{j+1}+a_{j-1},\quad j\in\mathbb{Z},
\label{NLSlattice}
\end{equation}
where the dot denotes the derivative with respect to the slow time
$\tau = \epsilon t$ and the parameter $\nu = \rho/\epsilon$ is defined
in the asymptotic range $\epsilon \ll \nu \leq 1$.

With the account of the dNLS equation (\ref{NLSlattice}), the
leading-order solution (\ref{envelope-approximation}) substituted into
the dKG equation (\ref{KG}) produces the residual terms in the form
\begin{equation}
{\rm Res}_j(t) := \rho \left( a_j^3 e^{3 i t} + \bar{a}_j^3 e^{-3it} \right) +
\epsilon^2 \left( \ddot{a}_j e^{i t} + \ddot{\bar{a}}_j e^{-it} \right).
\label{residual}
\end{equation}
The second residual term is resonant but occurs in the higher order
$\mathcal{O}(\epsilon^2)$, which is not an obstacle in the
justification analysis. The first residual term is non-resonant but it
occurs at the leading order of $\mathcal{O}(\rho) \gg
\mathcal{O}(\epsilon^2)$.  Therefore, the first term needs to be
removed, which is achieved with the standard near-identity
transformation. Namely, we extend the leading-order approximation
(\ref{envelope-approximation}) to the form
\begin{equation}
X_j(t) = a_{j}(\epsilon t)e^{it} +
\bar{a}_{j}(\epsilon t) e^{-it} + \frac{1}{8} \rho  \left( a_j^3(\epsilon t) e^{3it} +
\bar{a}_j^3(\epsilon t) e^{-3 it} \right).
\label{leading-order}
\end{equation}
For simplicity, we do not mention that $X_j$ depends on $\epsilon$ and
$\rho$.  Substituting the approximation (\ref{leading-order}) into the
dKG equation (\ref{KG}), we obtain the new residual terms in the form
\begin{eqnarray}
\nonumber {\rm Res}_j(t) & := & \epsilon^2 \left( \ddot{a}_j e^{i t} +
\ddot{\bar{a}}_j e^{-it} \right) - \frac{1}{8} \epsilon \rho \left(
(a_{j+1}^3 + a_{j-1}^3)e^{3 i t} +
(\bar{a}_{j+1}^3 + \bar{a}_{j-1}^3)e^{-3 i t} \right)
\\ \nonumber & \phantom{t} & + \frac{3}{8} \rho^2 \left( a_j e^{i t} +
\bar{a}_j e^{-it} \right)^2 \left( a_j^3 e^{3 it} + \bar{a}_j^3 e^{-3
  it} \right) + \frac{9}{4} \epsilon \rho \left( i a_j^2 \dot{a}_j e^{3
  it} - i \bar{a}_j^2 \dot{\bar{a}}_j e^{-3 it} \right) \\ \nonumber &
\phantom{t} & + \frac{3}{64} \rho^3\left( a_j e^{i t} + \bar{a}_j
e^{-it} \right)\left( a_j^3 e^{3 it} + \bar{a}_j^3 e^{-3 it} \right)^2
+ \frac{1}{8} \epsilon^2 \rho \left(\ddot{a^3_j} e^{3 it} +
\ddot{\overline{a}^3_j} e^{-3 it}\right)\\ & \phantom{t} & + \frac{1}{512}
\rho^4 \left( a_j^3 e^{3 it} + \bar{a}_j^3 e^{-3 it} \right)^3.
\label{residual-next}
\end{eqnarray}
Note that all time derivatives of $a_j$ in the residual term (\ref{residual-next})
can be eliminated from the dNLS equation (\ref{NLSlattice})
provided that $\{ a_j \}_{j \in \mathbb{Z}}$ is a twice differentiable sequence with respect to time.
For all purposes we need, it is sufficient to consider the sequence space
$\ell^2(\mathbb{Z})$. Hence we denote
the sequence $\{a_j\}_{j \in \mathbb{Z}}$ in $l^2(\mathbb{Z})$ by ${\bf a}$.

The next results give preliminary estimates on global solutions of the
dNLS equation (\ref{NLSlattice}), the leading-order approximation
(\ref{leading-order}), and the residual term (\ref{residual-next}).

\begin{lemma}
For every ${\bf a}_0 \in \ell^2(\mathbb{Z})$ and every $\nu \in
\mathbb{R}$, there exists a unique global solution ${\bf a}(t)$ of the
dNLS equation (\ref{NLSlattice}) in $\ell^2(\mathbb{Z})$ for every $t
\in \mathbb{R}$ such that ${\bf a}(0) = {\bf a}_0$.  Moreover, the
solution ${\bf a}(t)$ is smooth in $t$ and $\| {\bf a}(t) \|_{\ell^2}
= \| {\bf a}_0 \|_{\ell^2}$.
\label{lemma-existence}
\end{lemma}

\begin{proof}
Local well-posedness and smoothness of the local solution ${\bf a}$
with respect to time variable $t$ follow from the contraction principle
applied to an integral version of the dNLS equation (\ref{NLSlattice}).
The contraction principle can be applied because
the discrete Laplacian operator is a bounded operator on
$\ell^2(\mathbb{Z})$, whereas $\ell^2(\mathbb{Z})$ is a Banach algebra
with respect to pointwise multiplication and the $\ell^2(\mathbb{Z})$ norm
is an upper bound for the $\ell^{\infty}(\mathbb{Z})$ norm of a sequence.
Global continuation of the local solution ${\bf a}$ follows
from the $\ell^2(\mathbb{Z})$ conservation of the dNLS equation
(\ref{NLSlattice}).
\end{proof}

\begin{lemma}
For every ${\bf a}_0 \in \ell^2(\mathbb{Z})$, there exists a positive
constant $C_X(\| {\bf a}_0 \|_{\ell^2})$ (that depends on $\| {\bf a}_0 \|_{\ell^2}$)
such that for every $\rho \in (0,1]$ and every $t \in \mathbb{R}$, the leading-order approximation
  (\ref{leading-order}) is estimated by
\begin{equation}
\| {\bf X}(t) \|_{\ell^2} + \| \dot{\bf X}(t) \|_{\ell^2} \leq C_X(\| {\bf
  a}_0 \|_{\ell^2}).
\label{leading-order-estimate}
\end{equation}
\label{lemma-leading-approximation}
\end{lemma}

\begin{proof}
The result follows from the Banach algebra property of $\ell^2(\mathbb{Z})$
and the global existence result of Lemma \ref{lemma-existence}.
\end{proof}

\begin{lemma}
Assume that $\rho \leq \epsilon$.  For every ${\bf a}_0 \in
\ell^2(\mathbb{Z})$, there exists a positive $\epsilon$-independent
constant $C_R(\| {\bf a}_0 \|_{\ell^2})$ (that depends on $\| {\bf
  a}_0 \|_{\ell^2}$) such that for every $\epsilon \in (0,1]$ and
  every $t \in \mathbb{R}$, the residual term in (\ref{residual-next})
  is estimated by
\begin{equation}
\| {\bf Res}(t) \|_{\ell^2} \leq C_R(\| {\bf a}_0 \|_{\ell^2}) \epsilon^2.
\label{residual-estimate}
\end{equation}
\label{lemma-residual}
\end{lemma}

\begin{proof}
The result follows from the Banach algebra property of
$\ell^2(\mathbb{Z})$, as well as from the global existence and
smoothness of the solution ${\bf a}(t)$ of the dNLS equation
(\ref{NLSlattice}) in Lemma \ref{lemma-existence}.
\end{proof}

\subsection{Justification of the dNLS equation on the dNLS time scale}

The main result of this section is the following justification
theorem.

\begin{theorem}
Assume that $\rho$ is defined in the asymptotic range $\epsilon^2 \ll
\rho \leq \epsilon$.  For every $\tau_0 > 0$, there is a small
$\epsilon_0 > 0$ and positive constants $C_0$ and $C$ such that for
every $\epsilon \in (0,\epsilon_0)$, for which the initial data
satisfies
\begin{equation}
\label{bound-initial}
\| \textbf{$\xi$}(0) - {\bf X}(0) \|_{l^2} + \| \dot{\textbf{$\xi$}}(0) - \dot{\bf
  X}(0) \|_{l^2} \leq C_0 \rho^{-1} \epsilon^2,
\end{equation}
the solution of the dKG equation (\ref{KG})
satisfies for every $t \in [-\tau_0 \rho^{-1},\tau_0 \rho^{-1}]$,
\begin{equation}
\label{bound-final}
\| \textbf{$\xi$}(t) - {\bf X}(t) \|_{l^2} + \| \dot{\textbf{$\xi$}}(t) - \dot{\bf
  X}(t) \|_{l^2} \leq C \rho^{-1} \epsilon^2.
\end{equation}
\label{theorem-justification}
\end{theorem}

\begin{remark}
If $\rho = \epsilon$, the justification result of Theorem
\ref{theorem-justification} guarantees that the dynamics of
small-amplitude oscillators follows closely the dynamics of the dNLS
equation (\ref{dnls-introduction}) on the dNLS time
scale $[-\tau_0,\tau_0]$ for the variable $\tau = \epsilon t$.
\end{remark}

\begin{remark}
If $\rho=\epsilon^{8/5}$, the error term in (\ref{bound-final}) satisfies the
$\mathcal{O}_{\ell^2}(\epsilon^{2/5})$ bound. The error term is
controlled on the longer time scale $[-\tau_0 \epsilon^{-3/5},\tau_0 \epsilon^{-3/5}]$
for the variable $\tau = \epsilon t$ of the dNLS equation
(\ref{NLSlattice}) with $\nu = \epsilon^{3/5}$.
\end{remark}

To develop the justification analysis, we write $\textbf{$\xi$}(t) = {\bf
  X}(t) + {\bf y}(t)$, where ${\bf X}(t)$ is the leading-order
approximation (\ref{leading-order}) and ${\bf y}(t)$ is the error
term. Substituting the decomposition into the lattice equation
(\ref{KG}), we obtain the evolution problem for the error term:
\begin{equation}
\ddot{y}_j + y_j + 3 \rho X_j^2 y_j + 3 \rho X_j y_j^2 + \rho y_j^3 -
\epsilon (y_{j+1} + y_{j-1}) + {\rm Res}_j = 0,\quad
j\in\mathbb{Z}, \label{KG-pert}
\end{equation}
where the residual term ${\bf Res}(t)$ is given by
(\ref{residual-next}) if ${\bf a}(t)$ satisfies the dNLS equation
(\ref{NLSlattice}).  Associated with the evolution equation
(\ref{KG-pert}), we also define the energy of the error term
\begin{equation}
\label{energy}
E(t) := \frac{1}{2} \sum_{j \in \mathbb{Z}} \left[ \dot{y}_j^2 + y_j^2
  + 3 \rho X_j^2 y_j^2 - 2\epsilon y_j y_{j+1} \right].
\end{equation}
For every $\epsilon \in \left(0,\frac{1}{4}\right)$, the energy $E(t)$
is coercive and controls the $\ell^2(\mathbb{Z})$ norm of the solution in the sense
\begin{equation}
\label{coercivity}
\| \dot{\bf y}(t) \|_{\ell^2}^2 + \| {\bf y}(t) \|^2_{\ell^2} \leq 4 E(t),
\end{equation}
for every $t$, for which the solution ${\bf y}(t)$ is defined.
The rate of change for the energy (\ref{energy}) is found from the
evolution problem (\ref{KG-pert}):
\begin{equation}
\frac{dE}{dt} = \sum_{j \in \mathbb{Z}} \left[ - \dot{y}_j {\rm Res}_j +
  3 \rho X_j \dot{X}_j y_j^2 - 3 \rho X_j y_j^2 \dot{y}_j - \rho y_j^3
  \dot{y}_j \right]. \label{energy-rate}
\end{equation}

Thanks to the coercivity (\ref{coercivity}), the Cauchy--Schwarz inequality, and the continuous
embedding of $\ell^2(\mathbb{Z})$ to $\ell^{\infty}(\mathbb{Z})$, we obtain
\begin{equation}
\left| \frac{dE}{dt} \right| \leq 2 E^{1/2} \left[ \| {\bf Res}(t)
  \|_{\ell^2} + 6 \rho E^{1/2} \| {\bf X}(t) \|_{\ell^2} \| \dot{\bf X}(t)
  \|_{\ell^2} + 12 \rho E \| {\bf X}(t) \|_{\ell^2} + 8 \rho E^{3/2}\right]. \label{energy-estimate}
\end{equation}
To simplify analysis, it is better to introduce the parametrization $E = Q^2$
and rewrite (\ref{energy-estimate}) in the equivalent form
\begin{equation}
\left| \frac{dQ}{dt} \right| \leq \| {\bf Res}(t) \|_{\ell^2} + 6 \rho
Q \| {\bf X}(t) \|_{\ell^2} \| \dot{\bf X}(t) \|_{\ell^2} + 12 \rho Q^2
\| {\bf X}(t) \|_{\ell^2} + 8 \rho Q^3. \label{Q-estimate}
\end{equation}
The energy estimate (\ref{Q-estimate}) is the starting point for the proof of Theorem
\ref{theorem-justification}.\\

\begin{proof1}{\bf of Theorem \ref{theorem-justification}.}
Let $\tau_0 > 0$ be fixed arbitrarily but independently of $\epsilon$
and assume that the initial norm of the perturbation term satisfies
the following bound
\begin{equation}
\label{initial-bound}
Q(0) \leq C_0 \rho^{-1} \epsilon^2,
\end{equation}
where $C_0$ is a positive $\epsilon$-independent constant and
$\epsilon \in \left( 0, \frac{1}{4} \right)$ is sufficiently small.  Note that the bound
(\ref{initial-bound}) follows from the assumption
(\ref{bound-initial}) and the energy (\ref{energy}) subject to the choice of constant $C_0$.

To justify the dNLS equation (\ref{NLSlattice}) on the time scale
$[-\tau_0 \rho^{-1},\tau_0 \rho^{-1}]$ for $t$, we define
\begin{equation}
T_0 := \sup\left\{ t_0 \in [0,\tau_0 \rho^{-1}] : \quad \sup_{t \in [-t_0,t_0]} Q(t) \leq C_Q \rho^{-1}
\epsilon^2 \right\}, \label{time-span}
\end{equation}
where $C_Q > C_0$ is a positive $\epsilon$-independent constant to be
determined below. By continuity of the solution in the
$\ell^2(\mathbb{Z})$ norm, it is clear that $T_0 > 0$.

By using Lemmas \ref{lemma-leading-approximation} and \ref{lemma-residual},
as well as the definition (\ref{time-span}), we write the energy estimate
(\ref{Q-estimate}) for every $t \in [-T_0,T_0]$ in the form
\begin{equation}
\left| \frac{dQ}{dt} \right| \leq C_R \epsilon^2 +
\rho \left( 6 C_X^2 + 12 C_X C_Q \rho^{-1} \epsilon^2 + 8 C_Q^2 \rho^{-2} \epsilon^4 \right) Q.
\label{Q-estimate-explicit}
\end{equation}
If $\epsilon > 0$ is sufficiently small and $\epsilon^2 \ll \rho$, for every $t \in
[-T_0,T_0]$, one can always find a positive
$\epsilon$-independent $k_0$ such that
\begin{equation}
6 C_X^2 + 12 C_X C_Q \rho^{-1} \epsilon^2 + 8 C_Q^2 \rho^{-2} \epsilon^4 \leq k_0.
\label{C-Q-bound}
\end{equation}
Integrating (\ref{Q-estimate-explicit}), we obtain
\begin{equation}
Q(t) e^{- \rho k_0 |t|} - Q(0) \leq \int_0^{|t|} C_R \epsilon^2
e^{-\rho k_0 t'} d t' \leq \frac{C_R \epsilon^2}{\rho k_0}. \label{Q-integral-estimate}
\end{equation}
By using (\ref{initial-bound}), we obtain for every $t \in [-T_0,T_0]$:
\begin{equation}
Q(t) \leq \rho^{-1} \epsilon^2 \left( C_0 + k_0^{-1} C_R \right) e^{k_0 \tau_0}.
\label{final-bound}
\end{equation}
Hence, we can define $C_Q := \left( C_0 + k_0^{-1} C_R \right) e^{k_0
  \tau_0}$ and extend the time interval in (\ref{time-span}) by
elementary continuation arguments to the full time span with
$T_0 = \tau_0 \rho^{-1}$. This completes justification of the
dNLS equation (\ref{NLSlattice}) in Theorem
\ref{theorem-justification}.
\end{proof1}

\subsection{Justification of the dNLS equation on the extended time scale}

Next, we justify the dNLS equation (\ref{NLSlattice}) on the extended
time scale
\begin{equation}
\label{time-interval-mod}
[-A |\log(\rho)| \rho^{-1},A |\log(\rho)| \rho^{-1}],
\end{equation}
for the variable $t$, where the positive constant $A$ is
fixed independently of $\epsilon$. The main
result of this section is the following justification theorem.

\begin{theorem}
Assume that there is $\alpha \in (0,1)$ such that
$\rho$ is defined in the asymptotic range
$$
\epsilon^{\frac{2}{1+\alpha}} \ll \rho \leq \epsilon.
$$
For every $A \in \left(0, k_0^{-1} \alpha \right)$,
where $k_0$ is defined in (\ref{C-Q-bound-mod}) below, there
is a small $\epsilon_0 > 0$ and positive constants $C_0$ and $C$ such
that for every $\epsilon \in (0,\epsilon_0)$, for which the initial
data satisfies
\begin{equation}
\label{bound-initial-extended}
\| \textbf{$\xi$}(0) - {\bf X}(0) \|_{l^2} + \| \dot{\textbf{$\xi$}}(0) - \dot{\bf
  X}(0) \|_{l^2} \leq C_0 \rho^{-1} \epsilon^2,
\end{equation}
the solution of the dKG equation (\ref{KG})
satisfies for every $t$ in the time span (\ref{time-interval-mod}),
\begin{equation}
\label{bound-final-extended}
\| \textbf{$\xi$}(t) - {\bf X}(t) \|_{l^2} + \| \dot{\textbf{$\xi$}}(t) - \dot{\bf
  X}(t) \|_{l^2} \leq C \rho^{-1-\alpha} \epsilon^{2}.
\end{equation}
\label{theorem-justification-extended}
\end{theorem}

\begin{remark}
If $\rho = \epsilon$, the extended time scale (\ref{time-interval-mod})
corresponds to the interval $[-A |\log(\epsilon)|,A |\log(\epsilon)|]$
for the variable $\tau = \epsilon t$ in the dNLS equation
(\ref{NLSlattice}), hence it extends to all times $\tau$ as
$\epsilon \to 0$.
\end{remark}

\begin{remark}
If $\rho = \epsilon^{8/5}$, then the error term in (\ref{bound-final-extended}) satisfies the
$\mathcal{O}_{\ell^2}(\epsilon^{2 (1-4\alpha) /5})$ bound, which is small if $\alpha \in
\left(0,\frac{1}{4}\right)$. The error term is controlled on the
longer time scale
$$
\left[-\tau_0 |\log(\epsilon)| \epsilon^{-3/5},\tau_0 |\log(\epsilon)| \epsilon^{-3/5}\right]
$$
for the variable $\tau = \epsilon t$ of the dNLS equation (\ref{NLSlattice}) with $\nu = \epsilon^{3/5}$.
\end{remark}

\begin{proof}
We use the same assumption (\ref{initial-bound}) on the initial norm
of the perturbation term.  To justify the dNLS equation
(\ref{NLSlattice}) on the time scale (\ref{time-interval-mod}) for
$t$, we define
\begin{equation}
  T_0^* := \sup\left\{ t_0 \in \left[0,A |\log(\rho)| \rho^{-1}\right] : \quad
  \sup_{t \in [-t_0,t_0]} Q(t) \leq C_Q \rho^{-1-\alpha} \epsilon^2 \right\},
  \label{time-span-mod}
\end{equation}
where $C_Q$ is a positive $\epsilon$-independent constant to be
determined below.

By using Lemmas \ref{lemma-leading-approximation} and \ref{lemma-residual},
as well as the definition (\ref{time-span-mod}), we write the energy estimate
(\ref{Q-estimate}) for every $t \in [-T_0^*,T_0^*]$ in the form
\begin{equation}
\left| \frac{dQ}{dt} \right| \leq C_R \epsilon^2 +
\rho \left( 6 C_X^2 + 12 C_X C_Q \rho^{-1-\alpha} \epsilon^2 + 8 C_Q^2 \rho^{-2(1+\alpha)} \epsilon^4 \right) Q.
\label{Q-estimate-mod}
\end{equation}
If $\epsilon > 0$ is sufficiently small and $\epsilon^2 \ll \rho^{1+\alpha}$, then for every $t
\in [-T_0^*,T_0^*]$, one can always find a positive
$\epsilon$-independent $k_0$ such that
\begin{equation}
6 C_X^2 + 12 C_X C_Q \rho^{-1-\alpha} \epsilon^2 + 8 C_Q^2 \rho^{-2(1+\alpha)} \epsilon^4  \leq k_0.
\label{C-Q-bound-mod}
\end{equation}
By integrating the energy estimate (\ref{Q-estimate-mod}) in the same way as is done in
(\ref{Q-integral-estimate}), we obtain
for every $t \in [-T_0^*,T_0^*]$:
\begin{eqnarray}
\nonumber
Q(t) & \leq & \rho^{-1} \epsilon^2  \left( C_0 + k_0^{-1} C_R \right) e^{k_0 A |\log(\rho)|}\\
& \leq & \rho^{-1-\alpha} \epsilon^2  \left( C_0 + k_0^{-1} C_R \right),
\label{final-bound-mod}
\end{eqnarray}
where the last bound holds because $k_0 A \in (0,\alpha)$.
Hence, we can define $C_Q := C_0 + k_0^{-1} C_R$ and extend the time
interval in (\ref{time-span-mod}) by elementary continuation arguments
to the full time span with $T_0^* = A |\log(\rho)|
\rho^{-1}$. This completes justification of the dNLS equation
(\ref{NLSlattice}) on the time scale (\ref{time-interval-mod}).
\end{proof}

\subsection{Approximations with the generalized dNLS equation}

Extensions of the justification analysis are definitely possible by
including more $\epsilon$-dependent terms into the dNLS equation
(\ref{NLSlattice}) and the leading-order approximation
(\ref{leading-order}), which makes the residual term
(\ref{residual-next}) to be as small as $\mathcal{O}(\epsilon^n)$ for
any $n \geq 2$. These extensions are not so important if $\epsilon^2
\ll \rho \leq \epsilon$ but they become crucial to capture the correct
balance between linear and nonlinear effects on the dynamics of
small-amplitude oscillators if $\rho \leq \epsilon^2$.

To illustrate these extensions, we show how to modify the
justification analysis in the asymptotic range $\epsilon^3 \ll \rho
\leq \epsilon^2$.  We use the same leading-order approximation
(\ref{leading-order}) in the form
\begin{equation}
X_j(t) = a_{j}(\epsilon t)e^{it} +
\bar{a}_{j}(\epsilon t) e^{-it} + \frac{1}{8} \rho  \left( a_j^3(\epsilon t) e^{3it} +
\bar{a}_j^3(\epsilon t) e^{-3 it} \right),
\label{leading-order-same}
\end{equation}
but assume that ${\bf a}(\tau)$ with $\tau = \epsilon t$ satisfy the generalized dNLS equation
\begin{equation}
2 i \dot{a}_j + 3 \epsilon \delta |a_j|^2 a_j = a_{j+1} + a_{j-1}
+ \frac{\epsilon}{4} \left( a_{j+2} + 2 a_j + a_{j-2}\right),\quad j\in\mathbb{Z}.
\label{NLSlattice-extended}
\end{equation}
Here we have introduced the parameter $\delta = \rho/\epsilon^2$ in the
asymptotic range $\epsilon \ll \delta \leq 1$.  Substituting
(\ref{leading-order-same}) and (\ref{NLSlattice-extended}) into the
dKG equation (\ref{KG}), we obtain modifications of the residual
terms (\ref{residual-next}) in the form
\begin{eqnarray}
\nonumber {\rm Res}_j(t) & := & \frac{1}{4} \epsilon^2 \left( 4
\ddot{a}_j + a_{j+2} + 2a_j + a_{j-2} \right) e^{i t} + \frac{1}{4} \epsilon^2
\left( \ddot{\bar{a}}_j + \bar{a}_{j+2} + 2 \bar{a}_j + \bar{a}_{j-2}  \right) e^{-it}
\\ \nonumber & \phantom{t} & - \frac{1}{8} \epsilon \rho \left(
(a_{j+1}^3 + a_{j-1}^3)e^{3 i t} +
(\bar{a}_{j+1}^3 + \bar{a}_{j-1}^3)e^{-3 i t} \right)
\\ \nonumber & \phantom{t} & + \frac{3}{8} \rho^2 \left( a_j e^{i t} +
\bar{a}_j e^{-it} \right)^2 \left( a_j^3 e^{3 it} + \bar{a}_j^3 e^{-3
  it} \right) + \frac{9}{4} \epsilon \rho \left( i a_j^2 \dot{a}_j e^{3
  it} - i \bar{a}_j^2 \dot{\bar{a}}_j e^{-3 it} \right) \\ \nonumber &
\phantom{t} & + \frac{3}{64} \rho^3\left( a_j e^{i t} + \bar{a}_j
e^{-it} \right)\left( a_j^3 e^{3 it} + \bar{a}_j^3 e^{-3 it} \right)^2
+ \frac{1}{8} \epsilon^2 \rho \left( \ddot{a^3_j} e^{3 it} +
\ddot{\overline{a}^3_j} e^{-3 it} \right)\\ & \phantom{t} & + \frac{1}{512}
\rho^4 \left( a_j^3 e^{3 it} + \bar{a}_j^3 e^{-3 it} \right)^3.
\label{residual-next-extended}
\end{eqnarray}
By using the extended dNLS equation
(\ref{NLSlattice-extended}), we realize that the residual terms of the
$\mathcal{O}_{\ell^2}(\epsilon^2)$ order are canceled and the
residual term in (\ref{residual-next-extended}) enjoys the improved estimate
\begin{equation}
\| {\bf Res}(t) \|_{\ell^2} \leq C_R(\| {\bf a}_0 \|_{\ell^2}) \epsilon^3,
\label{residual-estimate-extended}
\end{equation}
compared with the previous estimate (\ref{residual-estimate}). As a result, the
justification analysis developed in the proof of Theorems \ref{theorem-justification}
and \ref{theorem-justification-extended} holds verbatim and results in the following theorems.

\begin{theorem}
Assume that $\rho$ is defined in the asymptotic range $\epsilon^3 \ll
\rho \leq \epsilon^2$.  For every $\tau_0 > 0$, there is a small
$\epsilon_0 > 0$ and positive constants $C_0$ and $C$ such that for
every $\epsilon \in (0,\epsilon_0)$, for which the initial data
satisfies
\begin{equation}
\label{bound-initial-generalized}
\| \textbf{$\xi$}(0) - {\bf X}(0) \|_{l^2} + \| \dot{\textbf{$\xi$}}(0) - \dot{\bf
  X}(0) \|_{l^2} \leq C_0 \rho^{-1} \epsilon^3,
\end{equation}
the solution of the dKG equation (\ref{KG})
satisfies for every $t \in [-\tau_0 \rho^{-1},\tau_0 \rho^{-1}]$,
\begin{equation}
\label{bound-final-generalized}
\| \textbf{$\xi$}(t) - {\bf X}(t) \|_{l^2} + \| \dot{\textbf{$\xi$}}(t) - \dot{\bf
  X}(t) \|_{l^2} \leq C \rho^{-1} \epsilon^3.
\end{equation}
\label{theorem-justification-generalized}
\end{theorem}

\begin{theorem}
Assume that there is $\alpha \in \left( 0, \frac{1}{2}\right)$ such that
$\rho$ is defined in the asymptotic range
$$
\epsilon^{\frac{3}{1+\alpha}} \ll \rho \leq \epsilon^2.
$$
There is $A_0 > 0$ such that for every $A \in \left(0, A_0 \right)$, there
is a small $\epsilon_0 > 0$ and positive constants $C_0$ and $C$ such
that for every $\epsilon \in (0,\epsilon_0)$, for which the initial
data satisfies
\begin{equation}
\label{bound-initial-extended-generalized}
\| \textbf{$\xi$}(0) - {\bf X}(0) \|_{l^2} + \| \dot{\textbf{$\xi$}}(0) - \dot{\bf
  X}(0) \|_{l^2} \leq C_0 \rho^{-1} \epsilon^3,
\end{equation}
the solution of the dKG equation (\ref{KG})
satisfies for every $t$ in the time span (\ref{time-interval-mod}),
\begin{equation}
\label{bound-final-extended-generalized}
\| \textbf{$\xi$}(t) - {\bf X}(t) \|_{l^2} + \| \dot{\textbf{$\xi$}}(t) - \dot{\bf
  X}(t) \|_{l^2} \leq C \rho^{-1-\alpha} \epsilon^{3}.
\end{equation}
\label{theorem-justification-extended-generalized}
\end{theorem}

We note that ${\bf X}$ in Theorems \ref{theorem-justification-generalized} and \ref{theorem-justification-extended-generalized}
is defined by the leading-order approximation (\ref{leading-order-same}), whereas
${\bf a}$ satisfies the generalized dNLS equation (\ref{NLSlattice-extended}).
The time scales in Theorems \ref{theorem-justification-generalized} and \ref{theorem-justification-extended-generalized}
are appropriate for the generalized dNLS equation (\ref{NLSlattice-extended})
because $\delta \leq 1$ and $\epsilon \rho^{-1} \geq \epsilon^{-1}$.

\section{Justification of the dNLS equation with the normal form method}
\label{s:NF}

We now consider the dKG equation \eqref{KGlattice} on a finite chain of $2N+1$
oscillators under periodic boundary conditions. The finite dKG chain is
associated with the Hamiltonian $H = H_0 + H_1$, where
\begin{equation}
  \label{e.H}
H_0:=\frac12\sum_{j=-N}^{N}
  \quadr{y^2_j + x^2_j - 2\eps x_{j+1}x_j}\ ,\qquad H_1:= \frac{1}4
  \sum_{j=-N}^N x_j^4,
\end{equation}
subject to the periodic boundary conditions $x_{-N}=x_{N+1}$ and
$y_{-N}=y_{N+1}$.  Since $N$ can be considered arbitrary large in the
perturbation approach which follows, the finite dKG chain approximates the
infinite problem in the asymptotic sense, as $N \to \infty$. Although the
present theory can be adapted to the infinite lattice, we prefer to rely on
some already proved results for the finite dKG chain for the sake of brevity.

According to the previous result in \cite{PP14}, for any small coupling
$\eps$, there exists a canonical transformation $T_\Chi$ which puts the
Hamiltonian $H = H_0 + H_1$, with $H_0$ and $H_1$ in \eqref{e.H}, into an
extensive resonant normal form of order $r$
\begin{equation}
\label{hamiltonian-transformed}
  \Ham{r}{} = H_\Omega + \Zscr + {P^{(r+1)}}\ , \qquad\qquad
  \{H_\Omega,\Zscr\}=0 \ ,
\end{equation}
where $H_\Omega$ is the Hamiltonian for the system of $2N+1$ identical
oscillators of frequency $\Omega$ (which is the average of the linear
frequencies \cite{GPP13}), $\Zscr$ is a non-homogeneous polynomial of
order $2r+2$, $P^{(r+1)}$ is a remainder of order $2r+4$ and
higher, and $r$ grows as an inverse power of $\eps$.  Such a normal
form was shown to be well defined in a small ball
$B_{\rho^{1/2}}(0)\subset\Pscr$ of the phase space $\Pscr$, endowed
with the Euclidean norm (which becomes the $\ell^2(\Z)$ norm in the
limit $N\to\infty$), provided $r\rho^{1/2}\ll 1$. The linear part of
the Hamiltonian $H_\Omega = \Omega\rho$ is equivalent to the selected
squared norm (uniformly with $N$), thus the almost invariance of
$H_\Omega$ over times $|t|\sim (r^2\rho)^{-r-1}$ is easily derived
since $\dot H_\Omega=\{H_\Omega,P^{(r+1)}\}$.

Looking at the structure of $\Zscr$, the normal form $H_\Omega+\Zscr$ produces
a generalized dNLS equation, where all
oscillators are coupled to all neighbors and the coupling coefficients both
for linear and nonlinear terms decay exponentially with the distance between
sites.  To be more specific, $\Zscr$ can be split as the sum of homogeneous
polynomials $Z_0, Z_1, ..., Z_r$, where $Z_0$ is quadratic, $Z_1$ is quartic,
and $Z_r$ is of the order $2r+2\geq 4$. Each of these
homogenous polynomials can be developed in powers of the coupling coefficient $\eps$,
where the term of order $\eps^m$ is responsible for the coupling between lattice sites
separated by the distance $m$.  The key ingredient to obtain the normal form is the
preservation of the translation invariance (called cyclic symmetry in
\cite{GPP13,PP14}), which also allows us to produce estimates that are uniform with $N$.

If we limit to $r=1$, the transformed Hamiltonian (\ref{hamiltonian-transformed}) reads
\begin{displaymath}
  \Ham{1}{} = \K + P^{(2)}\ ,\qquad \K := H_\Omega + Z_0 + Z_1\ ,
\end{displaymath}
where the quadratic and quartic polynomials $Z_0$ and $Z_1$ include
all-to-all interactions, exponentially decaying with $\eps$.  Hence,
$\K$ represents Hamiltonian of the generalized dNLS equation.
If we truncate both $Z_0$ and $Z_1$ at the leading order in $\eps$, we
recover Hamiltonian of the usual dNLS equation.

In Section 3.1, we introduce some definitions. The linear transformation is analyzed in
Section 3.2. The nonlinear normal form transformation is performed in Section 3.3.
Approximations with the usual dNLS equation are obtained in Section~\ref{ss:dNLS.gronw}.
Approximations with the generalized dNLS equation are discussed in Section 3.5.

\subsection{Some definitions}

We start recalling some definitions which allow us to characterize the structure
of the normal form (see also \cite{GPP12,GPP13,PP14,PP15}).

\paragraph{Cyclic symmetry:}
We formalize the {translational invariance} of the model \eqref{e.H} by using
the idea of \emph{cyclic symmetry}. The \emph{cyclic permutation} operator
$\tau$ is defined as
\begin{equation}
  \label{e.perm}
  \tau(x_{-N},\ldots,x_N) = (x_{-N+1},\ldots,x_N,x_{-N}).
\end{equation}
This operator can be applied separately to the variables $x$ and $y$. We
extend the action of this operator on the space of functions as $\bigl(\tau
f\bigr)(x,y) = f(\tau x,\tau y)$.

\begin{definition}
  \label{d.cs}
  We say that a function $F$ is \emph{cyclically symmetric} if $\tau F = F$.
\end{definition}

We introduce now an operator, indicated by an upper index $\oplus$,
acting on functions: given a function $f$, a new function $F=
f^{\oplus}$ is constructed as
\begin{equation}
  \label{e.cycl-fun}
  F= f^{\oplus} := \sum_{l=-N}^{N} \tau^l f \ .
\end{equation}
We say that $f^{\oplus}(x,y)$ is generated by the \emph{seed}
$f(x,y)$. Our convention is to denote the cyclically symmetric
functions by capital letters and their seeds by the corresponding
lower case letters.

\paragraph{Polynomial norms:}
Since we are interested in homogeneous polynomials (due to the use of Taylor
expansion), we introduce the following norms.

\begin{definition}
  \label{d.poly.norm}
  Let $f(x,y)=\sum_{|j|+|k|=s} f_{j,k} x^j y^k$ be a homogeneous
  polynomial of degree $s$ in $x,\,y$ and $F=f^\oplus$.  Given a
  positive radius $R$, we define the polynomial norm of $f$ by
  \begin{equation}
    \label{e.polinorm}
    \|f\|_R := R^s \sum_{|j|+|k|=s} |f_{j,k}|\ .
  \end{equation}
  Correspondingly, the extensive norm of $F$ is given by
  \begin{equation}
    \label{e.norm-germ}
    \bigl\|F\bigr\|^{\oplus}_R = \|f\|_R\ ,
  \end{equation}
\end{definition}

\paragraph{Vector fields:}
Let $F$ be an extensive Hamiltonian with seed $f$; we will make use of the
notation $X_F$ to indicate the associated Hamiltonian vector field $J\nabla
F$, with $J$ given by the standard Poisson structure. The Hamiltonian vector
field inherits, in a particular form, the cyclic symmetry: indeed it holds
true (see \cite{PP14,PP15})
\begin{equation}
  \label{e.seme.campo}
    \partial_{x_j}F = \tau^{j} \partial_{x_0}F, \quad
    \partial_{y_j}F = \tau^{j} \partial_{y_0}F, \quad
  j=-N,\ldots,N \ .
\end{equation}
As a result, a possible (but not unique) choice for its seed turns out to be
the couple $(\partial_{y_0}F, -\partial_{x_0}F)$. This fact allows us to
define in a reasonable and consistent way the following norm
\begin{equation}
  \label{e.def1}
  \ncamp{X_F}_R := \norm{\partial_{y_0}F}_R+\norm{\partial_{x_0}F}_R \ .
\end{equation}

\paragraph{Interaction range and centered alignment:}
Let us now consider monomials $x^j y^k$ in multi-index notations for $(j,k)$.

\begin{definition}
Given the exponents $(j,k)$, we define the support $S(x^jy^k)$ of the
monomials $x^j y^k$ and the interaction distance $\ell(x^jy^k)$ as follows:
\begin{equation}
  \label{e.supp}
  S(x^jy^k) = \{l\>:\>j_l\neq0 {\rm\ or\ } k_l\neq0 \}\ ,\quad
  \ell(x^jy^k) = \diam\bigl(S(x^jy^k)\bigr) \ .
\end{equation}
\end{definition}
We want to stress that, differently from what has been developed in
\cite{GPP12,GPP13}, it is possible to implement tha same construction by
asking the seeds of all the functions to be \emph{centered aligned}, according
to the following definition \cite{PP14}.

\begin{definition}
  Let $F=f^\oplus$ be a cyclically symmetric functions, with $f$ depending on
  $2N+1$ variables, $f=f(x_{-N},\ldots,x_0,\ldots,x_N)$. The seed $f$ is said
  centered aligned if it admits the decomposition
  \begin{equation}
    \label{e.decomp.cent}
    f=\sum_{m=0}^{N}f^{(m)}\ ,\qquad\qquad S(f^{(m)})\subseteq
    [-m,\ldots,m]\ .
  \end{equation}
\end{definition}

\paragraph{Exponential decay:}
In order to formalize and control the interaction range, we introduce
\begin{definition}
  \label{d.Cf_class}
  The seed $f$ of a function $F$ is said to be of class $\Dscr(C_f,\mu)$ if
  there exist two positive constants $C_f$ and $\mu<1$ such that for any
  centered aligned component $f^{(m)}$ it holds
  \begin{displaymath}
    \norm{f^{(m)}} \leq C_f \mu^m\ ,\qquad m=0,\ldots,N\ .
  \end{displaymath}
\end{definition}

\subsection{Linear transformation}

Let us focus on the harmonic part $H_0$ of the Hamiltonian $H$. From \eqref{e.H}, $H_0$ can be
written as the quadratic form
\begin{equation}
  \label{e.H0}
  H_0(x,y) = \frac12 y\cdot y + \frac12 Ax\cdot x\,
\end{equation}
where $A$ is a circulant and symmetric matrix given by
\begin{equation}
  \label{e.mu}
  A := \Id - \eps(\tau + \tau^{\top})\ .
\end{equation}
Here $\tau=(\tau_{ij})$ is the matrix representing the cyclic permutation
\eqref{e.perm}, i.e. with $\tau_{ij}=\delta_{i,j+1\>({\rm mod}\, 2N+1)}$ using
the Kronecker's delta notation.

\begin{proposition}
  \label{p.1}
  For every $\eps \in (0,\frac12)$ the canonical linear transformation $q=A^{1/4} x$,
  $p=A^{-1/4}y$ transforms the quadratic Hamiltonian $H_0$ to the quadratic normal form
  \begin{equation}
    \label{e.lindecomp}
    H^{(0)} = H_\Omega + Z_0\ ,\qquad\qquad  \Poi{H_\Omega}{Z_0}=0 \ ,
  \end{equation}
  where $H_\Omega = h_\Omega^\oplus$ and $Z_0=\zeta_0^\oplus$ are
  cyclically symmetric polynomials, with centered aligned seeds
  $h_\Omega$ and $\zeta_0$ of the form
  \begin{equation}
    \label{e.H0.seeds}
      h_\Omega = \frac{\Omega}{2}(q_0^2+p_0^2)
      \end{equation}
      and
      \begin{equation}
      \label{e.Zeta}
      \zeta_0 =\sum_{m=1}^N\zeta_0^{(m)},           \quad
      \zeta_0^{(m)} = b_m\quadr{q_0(q_m+q_{-m})+p_0(p_m+p_{-m})}.
  \end{equation}
  Here $\Omega$ and $b_m$ are defined by
  \begin{equation}
    \label{e.b_m.Omega}
    \Omega :=
    \frac1{2N+1}\sum_{j=-N}^{N+1}\omega_j\ , \quad b_m := \tond{A^{1/2}}_{1,m+1}\ ,
  \end{equation}
  whereas $\omega_j$ are the frequencies of the normal modes of $H_0$.  Moreover,
  there exists a suitable positive constant $C_{\zeta_0}$ such that each component
  $\zeta_0^{(m)}$ satisfies the exponential decay
  \begin{displaymath}
    \norm{\zeta_0^{(m)}}\leq C_{\zeta_0} (2\eps)^m\ ,
  \end{displaymath}
  hence $\zeta_0\in \Dscr\bigl(C_{\zeta_0},2\eps\bigr)$.
\end{proposition}

\begin{proof}
We give here only few ideas to grasp the exponential decay of the
all-to-all interactions due to the linear transformation. After
applying $q=A^{1/4} x$, $p=A^{-1/4}y$, we have
\begin{equation}
  \label{e.H0.decomp}
  H_0 = \frac12 p^{\top} A^{1/2} p + \frac12 q^{\top} A^{1/2} q.
\end{equation}
By defining $T:=\tau+\tau^{\top}$, one can rewrite $A^{1/2}$ as
\begin{displaymath}
  A^{1/2} = (\Id - \eps T)^{1/2} = \sum_{l = 0}^{\infty}
  \binom{1/2}{l}(-\eps)^l T^l\ .
\end{displaymath}
In order to obtain the decomposition \eqref{e.lindecomp}, we separate
the diagonal part from the off-diagonal part $A^{1/2} = \Omega \; \Id
+ B$ and insert this decomposition into \eqref{e.H0.decomp}. The
exponential decay $(2\eps)^m$ comes from the observation that
$\tond{T^l}_{1,m+1}=0$ for all $0\leq l< m$ and from the estimate
$|\tond{T^m}_{1,m+1}|\leq 2^m$. One can restrict to consider only the
first raw due to the circulant nature of all the matrices involved
(for all details see Appendix 6.1.1 in \cite{GPP13}).
\end{proof}

\begin{proposition}
\label{l.H1}
Under the linear transformation in Proposition \ref{p.1}, the quartic part
$H_1$ given in (\ref{e.H}) is cyclically symmetric with a centered aligned seed
$H_1 = h_1^\oplus$ given by
\begin{equation}
  \label{e.H1.seed}
h_1 = \sum_{m=0}^N h_1^{(m)}\ .
\end{equation}
Moreover, there exists a suitable positive constant $C_{h_1}$ such that each component
$h_1^{(m)}$ satisfies the exponential decay
\begin{displaymath}
  \norm{h_1^{(m)}}\leq C_{h_1} (2\eps)^m\ ,
\end{displaymath}
hence $h_1\in \Dscr\bigl(C_{h_1},2\eps\bigr)$.
\end{proposition}

We can translate Propositions \ref{p.1} and \ref{l.H1} by saying that
in a suitable set of coordinates, the coupling part of the quadratic
Hamiltonian $H_0$ shows all-to-all linear interactions, with an
exponentially decaying strength with respect to the distance between
the sites. Such a linear transformation introduces similar all-to-all
interactions also in the quartic Hamiltonian $H_1$.  Moreover, in the
new coordinates $q_j$, the seed $h_1$ of the quartic term has the same
exponential decay as the seed $\zeta_0$ of the quadratic term does.

\subsection{First-order nonlinear normal form transformation}

With the Hamiltonian $H$ transformed by means of Propositions \ref{p.1} and
\ref{l.H1} into the form
\begin{equation}
  \label{e.H.lintrs}
  H = H_\Omega + Z_0 + H_1\ ,
\end{equation}
we are now ready to state the (first-order) normal form theorem.  This
first-order theorem represents the easiest formulation of the more
generic Theorem~1 of \cite{PP14}. The idea is to perform, by using the
Lie transform algorithm explained in \cite{G03}, one normalizing step,
provided $\eps$ is small enough. Moreover, the normalizing canonical
transformation is well defined in a (small) neighborhood
$B_{\rho^{1/2}}$ of the origin, where $\rho$ is sufficiently small.

\begin{theorem}
\label{prop.gen}
Consider the Hamiltonian $H=h^{\oplus}_{\Omega}+\zeta^{\oplus}_0 +
h^{\oplus}_1$ with seeds $h_{\Omega},\,\zeta_0,\,h_1$, in \eqref{e.H0.seeds}, \eqref{e.Zeta},
and \eqref{e.H1.seed}.  There exist positive $\gamma$, $\eps_{*}<\frac12$ and
$C_*$ such that for every $\eps \in (0,\eps_*)$,
there exists a generating function $\Chi_1=\chi^{\oplus}_1$ of a Lie transform
such that $T_{\Chi_1}\Ham{1}{} = H$, where $\Ham{1}{}$ is a cyclically
symmetric function of the form
\begin{equation}
  \label{e.Ham.r}
  \Ham{1}{} = H_\Omega + Z_0 + Z_1 + {P^{(2)}}\ ,
\end{equation}
with $0 = \{ H_\Omega ,Z_0 \} = \{ H_\Omega ,Z_1 \}$, whereas
$Z_1 = \zeta_1^\oplus$ is a polynomial of degree four with a seed $\zeta_1$ is of class
$\Dscr\left(C_{h_1}, 2\eps\right)$, and $P^{(2)}$ is a remainder that includes
terms of degree equal or bigger than six. Moreover, if the smallness condition on the energy
\begin{equation}
  \label{e.R.sm1}
  \rho<\rho_*:= \frac{1}{96 (1+e) C_*}\ ,
\end{equation}
is satisfied, then the following statements hold true:
\begin{enumerate}
\item $\Chi_1$ defines an analytic canonical transformation on the domain
  $B_{\frac23 {\rho^{1/2}}}$ such that
  \begin{displaymath}
    B_{\frac{1}{3} \rho^{1/2}}\subset T_{\Chi_1} B_{\frac23 {\rho^{1/2}}} \subset B_{\rho^{1/2}}\qquad\qquad
    B_{\frac{1}{3} \rho^{1/2}}\subset T_{\Chi_1}^{-1} B_{\frac23 {\rho^{1/2}}} \subset B_{\rho^{1/2}}\ .
  \end{displaymath}
  Moreover, the deformation of the domain $B_{\frac23 {\rho^{1/2}}}$ is controlled by
  \begin{equation}
    \label{e.def.Tchi}
    z\in B_{\frac23 {\rho^{1/2}}}\qquad\Rightarrow\qquad
    \norm{T_{\Chi_1}(z)-z}\leq 4^4 C_* \rho^{3/2}\ ,\qquad
    \norm{T^{-1}_{\Chi_1}(z)-z}\leq 4^4 C_* \rho^{3/2}\ .
  \end{equation}
\item the remainder is an analytic function on $B_{\frac23 {\rho^{1/2}}}$, and it is
  represented by a series of cyclically symmetric homogeneous polynomials
  $\Ham{1}{s}$ of degree $2s+2$
  \begin{equation}
    \label{e.rem.r}
    P^{(2)} = \sum_{s = 2}^{\infty} \Ham{1}{s} \qquad \Ham{1}{s} =
    \tond{h^{(1)}_s}^{\oplus}\ , \qquad h^{(1)}_s \in \Dscr(2\tilde
    C_*^{s-1}C_{h_1},\sqrt{2\eps}) \ .
  \end{equation}
\end{enumerate}
\end{theorem}

The interval $(0,\eps_*)$ with $\eps_* < \frac{1}{2}$ comes from the inequality
$$
f(\eps):=\tond{\frac{3\Omega}{64C_{\zeta_0}}} \frac{(1-2\eps)
  \quadr{1-(2\eps)^{\frac34}}}{\sqrt{2\eps}}>1
$$
(see for reference formula (33) in \cite{PP15}), and the constants $C_*$ and $\gamma$ can
be written as
\begin{equation}
  \label{const.prop.gen}
  C_* = \frac{4C_{h_1}}{3\gamma(1-2\eps)\quadr{1-(2\eps)^{\frac34}}}
\end{equation}
and
\begin{equation}
  \gamma = 2\Omega\left(1-\frac1{2f(\eps)}\right)
  \quad\Rightarrow\quad
  \Omega<\gamma < 2\Omega.
\end{equation}
Since $\eps$ is sufficiently smaller than $\frac{1}{2}$, the constants $C_*$ is
essentially independent on $\eps$, i.e.
$$
C_* = \mathcal{O}\tond{\frac{C_{h_1}}{\Omega}},
$$
which implies that the same holds true for the threshold $\rho_*$ so that
\begin{equation}
  \label{e.R*.appr}
  \rho_* \approx {\frac{2\Omega}{3C_{h_1}(1+e)}}\ .
\end{equation}

\subsection{Approximation with the dNLS equation}
\label{ss:dNLS.gronw}

We apply here the normal form transformation of Theorem \ref{prop.gen} in
order to approximate the Cauchy problem $\dot z = \{H,z\}$ of the finite dKG equation (\ref{KGlattice})
with a small initial datum $z_0$. Let us denote with $\K := H_\Omega + Z_0 + Z_1$ the
normal form part of the Hamiltonian $H^{(1)} = \K + P^{(2)}$ in formula
\eqref{e.Ham.r}. Since $Z_0$ and $Z_1$ have centered aligned seeds with the
exponential decay, see decompositions \eqref{e.Zeta} and \eqref{e.H1.seed}, we have
\begin{equation}
  \label{e.K.def}
  Z_0 = \sum_{m=1}^N Z_0^{(m)} \ ,
  \qquad\qquad Z_0^{(m)} := \tond{\zeta_0^{(m)}}^\oplus
  \end{equation}
  and
  \begin{equation}
    \label{e.K.def-1}
  Z_1 = \sum_{m=0}^N Z_1^{(m)} \ ,
  \qquad\qquad Z_1^{(m)} := \tond{\zeta_1^{(m)}}^\oplus\ .
\end{equation}
Note that the expansion for $Z_0$ starts at $m=1$, while $Z_1$ starts with
$m=0$. By truncating the $\eps$ expansion of each normal form term $Z_j$ at
their leading orders, we define the \emph{effective normal form Hamiltonian}
$\Keff$ as
\begin{equation}
  \label{e.Keff}
  \Keff := H_\Omega + Z_0^{(1)} + Z_1^{(0)}\ ,\qquad\qquad
  \Kres := \K - \Keff\ .
\end{equation}
As already stressed in \cite{PP14}, the truncated normal form $\Keff$
represents Hamiltonian of the dNLS equation. In complex coordinates $\psi_j=(q_j +
{\rm i} p_j)/\sqrt2$, Hamiltonian $\Keff$ reads as
\begin{equation}
  \label{e.dNLS.Ham}
  \Keff = (\Omega+2b_1)\sum_j|\psi_j|^2
  -b_1\sum_j|\psi_{j+1}-\psi_j|^2 +\frac{3}{8}\sum_j|\psi_j|^4\ ,
\end{equation}
where $b_1=\mathcal{O}(\eps)<0$ is the same as in the expression \eqref{e.b_m.Omega} of
Proposition~\ref{p.1}. The corresponding dNLS equation is
\begin{equation}
  \label{e.dNLS.eqs}
        {\rm i}\dot\psi_j = \derp{\Keff}{\overline\psi_j} = \Omega\psi_j + b_1
        (\psi_{j+1}+\psi_{j-1}) + \frac34 \psi_j|\psi_j|^2\ ,
\end{equation}
and it has the same structure as the dNLS equation (\ref{NLSlattice}).

We denote with $z(t)$ the evolution of the dKG transformed Hamiltonian
$\K+P^{(2)}$, with $z_a(t)$ the evolution of the dNLS model $\Keff$ and
consequently with $\delta(t)$ the error
\begin{equation}
  \label{error-delta}
  \delta(t) := z(t) - z_a(t)\ .
\end{equation}
The two time scales over which we control the error of the approximation are
given by
\begin{equation}
  \label{e.T.scales}
  T_0 := \frac1{\rho} \ ,\qquad\qquad T_0^* :=
  \frac\alpha{\kappa_0\rho}\ln\tond{\frac{1}\rho} \ ,
\end{equation}
where $\alpha\in(0,1)$ is an arbitrary parameter, and $\kappa_0 =
\mathcal{O}(C_{h_1})$ is given in \eqref{e.kappa0}. Similar
definitions are used in \eqref{time-span} and \eqref{time-span-mod},
in the proof of Theorems \ref{theorem-justification} and
\ref{theorem-justification-extended}.

\begin{theorem}
  \label{p.gronw.dNLS}
  Let us take
  $\rho$ fulfilling \eqref{e.R.sm1} and $\eps \in (0,\eps_*)$
  as in Theorem \ref{prop.gen}. Let us first consider the two independent
  parameters $\rho$ and $\eps$ in the regime $\eps^2\ll \rho\leq\eps
  $.  Then, there exists a positive constant $\Cst$ independent of $\rho$ and
  $\eps$ such that for any initial datum $z_0\in B_{\frac23\rho^{1/2}}$
  with $\norm{\delta_0}\leq \rho^{-1/2}\eps^2$, the following holds
  true
  \begin{equation}
    \label{e.est.error.dnls.1}
    \norm{\delta(t)} \leq
      \Cst \rho^{-1/2}\eps^2\ ,\qquad\quad |t|\leq T_0 \ .
  \end{equation}
  Let us now consider the two independent parameters $\rho$ and $\eps$
  in the regime $\eps^{\frac{2}{1+\alpha}}\ll \rho\leq\eps $, where
  $\alpha\in(0,1)$ is arbitrary.  Then, there exists a positive constant $\Cst$
  independent of $\rho$ and $\eps$ such that for any initial datum
  $z_0\in B_{\frac23\rho^{1/2}}$ with $\norm{\delta_0}\leq
  \rho^{-1/2}\eps^2$, the following holds true
  \begin{equation}
    \label{e.est.error.dnls.2}
    \norm{\delta(t)} \leq
      \rho^{-1/2-\alpha}\eps^2, \quad |t|\leq T_0^*\ .
  \end{equation}
\end{theorem}

\begin{remark}
The upper bound for the error $\delta$ given in
\eqref{e.est.error.dnls.1} and \eqref{e.est.error.dnls.2} refers to
the time evolution of the normal form \eqref{e.dNLS.Ham} in the
transformed variables $\psi$, which are near-identity deformations of
the original variables $(x,y)$.  Since the transformation $T_\Chi$ is
Lipschitz, with a Lipschitz constant $L$ of order $L=\mathcal{O}(1)$,
the same bound of the error holds also in the original
coordinates. Thus, from the analytic point of view, the nonlinear
deformation of the variables does not affect the dependence of the
estimates on $\rho$ and $\eps$: only the constant $\Cst$ is changed by
the Lipschitz factor $L$.
\end{remark}

\begin{remark}
The above estimates are equivalent, both in terms of error smallness
and time scale, to the ones obtained in Theorems
\ref{theorem-justification} and \ref{theorem-justification-extended},
once the original variables $x_j={\rho^{1/2}}\xi_j$ are recovered.
\end{remark}

\begin{remark}
The requirement $\eps^2\ll\rho$ on the time scale $T_0$ is needed in
order to provide a meaningful approximation, which means that the
error is much smaller than the leading approximation $z_a(t)$
\begin{displaymath}
\norm{\delta(t)}\leq \rho^{-1/2}\eps^2 \ll
\rho^{1/2}\sim\norm{z_a(t)}\ .
\end{displaymath}
The same reason lies behind the requirement
$\eps^{\frac{2}{1+\alpha}}\ll\rho$ on the extended time scale $T_0^*$.
\end{remark}

\begin{proof1}{\bf of Theorem \ref{p.gronw.dNLS}.}
Following a standard approach (see a similar problem in \cite{BCP09}), we
first decompose the Hamiltonian $H = H_L + H_N$ in its quadratic  and quartic parts
$$
H_L := H_\Omega + Z_0, \quad H_N := Z_1+P^{(2)},
$$ so that $\Keff = H_L + H_N - P^{(2)}-\Kres$. Correspondingly, the
vector field is decomposed as $X_H = X_{H_L} + X_N$. Denote the linear operator
for $X_{H_L}$ by $\Lin$.  The equation of motions for $z(t)$ and $z_a(t)$
reads
\begin{equation}
  \label{e.res.1}
  \left\{
  \begin{aligned}
    \dot z &= \Lin z + X_{N}(z)\ ,
    \\
    \dot z_a &= \Lin z_a + X_N(z_a) - \Res(t)\ ,
  \end{aligned}\right.
  \qquad\qquad\text{with}\quad
  \Res(t) := X_{P^{(2)}}(z_a(t)) + X_{\Kres}(z_a(t)) \ .
\end{equation}
The error $\delta(t)$ defined by \eqref{error-delta} satisfies the equation
\begin{equation}
  \label{e.delta.ev}
  \dot \delta = \Lin \delta + \quadr{X_N(z_a+\delta) - X_N(z_a)}
  + \Res(t)\ ,
\end{equation}
whose solution, with the initial value $\delta_0$, is given by Duhamel formula
\begin{equation}
  \label{e.delta.sol}
  \delta(t) = e^{\Lin t}\delta_0 + e^{\Lin t}\int_0^t e^{-\Lin s}
  \quadr{X_N(z_a+\delta) - X_N(z_a) + \Res(s)}ds\ .
\end{equation}
Now, since $\Poi{H_L}{H_\Omega}=0$, one has that $\Lin$ is an isometry. This
allows to estimate
\begin{equation}
  \label{e.int.est}
  \norm{\delta(t)} \leq \norm{\delta_0} + \int_0^t
  \quadr{\norm{X_N(z_a(s)+\delta(s)) - X_N(z_a(s))} +
    \norm{\Res(s)}}ds\ .
\end{equation}
The second term in the r.h.s. can be estimated with the definition of the
residual and using the information that $z_a(t)$ preserves the norm, as a
consequence of the conservation of $H_\Omega$
\begin{displaymath}
  \norm{X_{P^{(2)}}(z_a(s))}\leq \Cst \frac{C_{h_1}  C_*
    \rho^{5/2}}{(1-\sqrt[4]{2\eps})^2} \ ,\qquad
  \norm{X_{\Kres}(z_a(t))}\leq \Cst\frac{\quadr{ C_{\zeta_0}
      \rho^{1/2}\eps^2 + C_{h_1} \rho^{3/2}\eps}}{(1-2\eps)^2} \ ,
\end{displaymath}
where the two contributions in the second inequality come from the truncation
of $Z_0$ and $Z_1$ respectively. Thus, we obtain
\begin{equation}
  \label{e.est.res}
  \norm{\Res(s)} \leq \Cst \frac{\rho^{1/2}}{(1-\sqrt[4]{2\eps})^2}
  \quadr{C_{\zeta_0}\eps^2 + C_{h_1}\rho\eps + C_{h_1} C_*
    \rho^2}\ .
\end{equation}
On the other hand, if
\begin{equation}
  \label{e.delta.zeta}
  \norm{\delta} \ll \norm{z_a}\sim \rho^{1/2}\ ,
\end{equation}
then the increment of the nonlinear field can be well approximated by
\begin{displaymath}
  \norm{X_N(z_a(s)+\delta(s)) - X_N(z_a(s))} \leq
  \norm{X_N'(\zeta_a)}\norm{\delta}\ ,
\end{displaymath}
where
\begin{displaymath}
  \zeta_a:=z_a+\lambda\delta\ ,\qquad \lambda\in (0,1)\ .
\end{displaymath}
If the smallness condition \eqref{e.delta.zeta} for $\delta$ holds,
then $\norm{\zeta_a}\sim \rho^{1/2}$, which implies
\begin{displaymath}
  \norm{X_N(z_a(s)+\delta(s)) - X_N(z_a(s))} \leq
  \norm{X_N'}_\rho\norm{\delta}\ .
\end{displaymath}
By using the decomposition $X_N' = X_{Z_1}'+X_{P^{(2)}}'$ it is possible to
obtain
\begin{equation}
  \label{e.X.N.diff}
  \norm{X_N'}_{\rho^{1/2}} \leq
  \Cst_1\frac{C_{h_1}}{(1-\sqrt[4]{2\eps})^2} \rho \ .
\end{equation}
By inserting \eqref{e.est.res} and \eqref{e.X.N.diff} into
\eqref{e.int.est}, one gets a typical Gronwall-like integral
inequality (see, e.g., Lemma A.2 in \cite{BCP09}), which provides
the time-dependent upper bound
\begin{align}
\nonumber
  \norm{\delta(t)} &\leq e^{\kappa_0\rho t}\norm{\delta_0} + \Cst
  \rho^{1/2} \quadr{\frac{\eps^2}{\rho} + \eps + C_*
    \rho}\tond{e^{\kappa_0\rho t} -1} \\
      \label{e.gronw.3}
      &\leq e^{\kappa_0\rho
    t}\rho^{-1/2}\eps^2 + \Cst\rho^{-1/2} \tond{e^{\kappa_0\rho t}
    -1}\quadr{\eps^2 + \rho\eps + C_*\rho^2}\ ,
\end{align}
where $\kappa_0$ provides an upper bound for $\norm{X_N'}_1$ in
\eqref{e.X.N.diff}
\begin{equation}
\label{e.kappa0}
\kappa_0 := \Cst_1\frac{C_{h_1}}{(1-\sqrt[4]{2\eps_*})^2} =
\mathcal{O}(C_{h_1})
\end{equation}
and $\Cst$ depends only on $\eps_*,\,C_{\zeta_0},\,C_{h_1}$. Then, the
bound \eqref{e.est.error.dnls.1} follows from the assumption $\rho\leq
\eps$.

The bound \eqref{e.est.error.dnls.2} is obtained similarly, just
replacing the time span $T_0^*$ in the above \eqref{e.gronw.3}, which
easily provides the factor $\rho^{-\alpha}$ in front of the estimate.
\end{proof1}

\subsection{Approximations with the generalized dNLS equation}
\label{ss:dNLS.ext.gronw}

The standard dNLS approximation is no more valid when $\eps^2\sim
\rho$. Indeed, in such a case, the contribution $\eps^2\rho^{-1}$
coming from the truncation of the linear field $X_{H_L}$ in
\eqref{e.gronw.3} is of order one, hence the error $\delta(t)$ can be
comparable with the approximation $z_a(t)$
\begin{displaymath}
  \norm{\delta(t)}\leq \Cst \rho^{1/2} \sim \norm{z_a(t)}\ .
\end{displaymath}
In such a regime, it is then necessary to include in the Hamiltonian
$\Keff$ at least the term $Z_0^{(2)}$, responsible for the
next-nearest neighbourhood linear interaction:
\begin{equation}
  \label{e.Keff.ext}
  \Keff := H_\Omega + Z_0^{(1)} +Z_0^{(2)} +Z_1^{(0)}\ .
\end{equation}
Following the same steps as in the proof of
Theorem~\ref{p.gronw.dNLS}, it is possible to prove the following
result, which is fully equivalent to Theorems
\ref{theorem-justification-generalized} and
\ref{theorem-justification-extended-generalized}.

\begin{theorem}
  \label{p.gronw.dNLS-ext}
  Let us
  take $\rho$ fulfilling \eqref{e.R.sm1} and $\eps \in (0,\eps_*)$
  as in Theorem \ref{prop.gen}. Let us first consider the two independent
  parameters $\rho$ and $\eps$ in the regime $\eps^{3}\ll \rho\leq
  \eps^{2}$.  Then, there exists a positive constant $\Cst$ independent on
  $\rho$ and $\eps$ such that for any initial datum $z_0\in
  B_{\frac23\rho}$ with $\norm{\delta_0}\leq \rho^{-1/2}\eps^3$, it
  holds true
\begin{equation}
  \label{e.est.error.dnls.ext.1}
  \norm{\delta(t)} \leq
    \Cst \rho^{-1/2}\eps^3\ ,\qquad\quad |t|\leq T_0 \ .
\end{equation}
  Let us now consider the two independent parameters $\rho$ and $\eps$
  in the regime $\eps^{\frac{3}{1+\alpha}}\ll \rho\leq\eps $, where
  $\alpha\in(0,1)$ is arbitrary.  Then, there exists a positive
  constant $\Cst$ independent of $\rho$ and $\eps$ such that for any
  initial datum $z_0\in B_{\frac23\rho^{1/2}}$ with
  $\norm{\delta_0}\leq \rho^{-1/2}\eps^2$, the following holds true
  \begin{equation}
    \label{e.est.error.dnls.ext.2}
    \norm{\delta(t)} \leq
      \rho^{-1/2-\alpha}\eps^3\ ,\qquad |t|\leq T_0^*\ .
  \end{equation}
\end{theorem}

The result of Theorem \ref{p.gronw.dNLS-ext} yields Hamiltonian for
the generalized dNLS equation:
\begin{equation}
  \label{e.dNLS.Ham.2}
  \Keff = (\Omega+2b_1+2b_2)\sum_j|\psi_j|^2
  -b_1\sum_j|\psi_{j+1}-\psi_j|^2-b_2\sum_j|\psi_{j+2}-\psi_j|^2
  +\frac{3}{8}\sum_j|\psi_j|^4\ ,
\end{equation}
where $b_2=\mathcal{O}(\eps^2)<0$ is the same as in the expression
\eqref{e.b_m.Omega} of Proposition~\ref{p.1}. The corresponding
generalized dNLS equation is
\begin{equation}
  \label{e.dNLS.eqs.2}
        {\rm i}\dot\psi_j = \Omega\psi_j + b_1(\psi_{j+1}+\psi_{j-1})
        + b_2 (\psi_{j+2}+\psi_{j-2}) + \frac34 \psi_j|\psi_j|^2\ ,
\end{equation}
which has the same structure as the generalized dNLS equation
\eqref{NLSlattice-extended}. Indeed, remembering that $\Omega$ in
\eqref{e.dNLS.eqs.2} also has an expansion in $\eps$, and that the
time variable is rescaled with $\eps$ in \eqref{NLSlattice-extended},
we can rewrite the right-hand-side of the generalized dNLS equation
\eqref{NLSlattice-extended} as follows:
\begin{displaymath}
  \frac{\eps}{2} a_j + (a_{j+1} + a_{j-1}) +
  \frac{\eps}4\tond{a_{j+2} + a_{j-2}} \ .
\end{displaymath}
This shows an $\eps$ correction to the nearest neighbour coefficient, which in
the normal form approach is embedded in the $\eps$-dependence of $\Omega$, $b_1$,
$b_2$ and of the transformed coordinates.

More generally, within the normal form approach, different regimes of
parameters can be treated with no efforts: once the requested scaling between
$\eps$ and $\rho$ is chosen, one easily derives the minimal, and also the
optimal, number of terms in the expansions of $Z_0$ and $Z_1$ to be
included. The estimates follows as easy as before. Here we give
the estimates for a general choice of truncation:
\begin{equation}
  \label{e.Keff.gen1}
  \Keff = H_\Omega + \sum_{j=1}^{l-1} Z_0^{(j)} + \sum_{j=0}^{n-1}
  Z_1^{(j)}\ ,
\end{equation}
where $N\geq l\geq 2$ and $N\geq n\geq 1$.  The error term $\delta$ is now estimated similarly to
\eqref{e.gronw.3} as follows:
\begin{equation}
  \label{e.gronw.5}
  \norm{\delta(t)} \leq e^{\kappa_0\rho t}\norm{\delta_0} + \Cst \rho^{1/2}
  \quadr{\frac{\eps^l}{\rho} + \eps^n + C_*
    \rho}\tond{e^{\kappa_0\rho t} -1}\ ,\qquad l\geq 2\ ,\qquad n\geq 1\ .
\end{equation}
Hence one can deal with all the regimes and with the desired error
precision in a compact and flexible way. The extension to higher order
terms in the nonlinearity would require further steps of the normal
form transformations, thus modifying thresholds $\eps_*$ and $\rho_*$,
following the general version of Theorem \ref{prop.gen} given in
\cite{PP14}.

\section{Applications of the dNLS equation}

We conclude the paper with a brief account of possible applications of the
dNLS equations (\ref{NLSlattice}) and (\ref{e.dNLS.eqs}), and their
generalizations \eqref{NLSlattice-extended} and \eqref{e.dNLS.eqs.2},
in the context of small-amplitude weakly coupled oscillators of the dKG
equation (\ref{KGlattice}).\\

{\bf 1. Existence of breathers.} Breathers of the dKG equation
(time-periodic solutions localized on the lattice) can be
constructed approximately by looking at the discrete solitons of the
dNLS equation (\ref{dnls-introduction}) in the form $a_j(\tau) = A_j e^{i \Omega \tau}$, where
$\Omega \in \mathbb{R} \backslash [-1,1]$ is defined outside the
spectral band of the linearized dNLS equation and ${\bf A} \in \ell^2(\mathbb{R})$ is
time-independent.

The limit $\epsilon \to 0$ is referred to as the anti-continuum limit
of the dKG equation (\ref{KGlattice}), when the breathers at a fixed energy are
continued uniquely from the limiting configurations supported on few
lattice sites \cite{MA94,PelSak}. Compared to the anti-continuum
limit, the dNLS approximation is very different, because the discrete
solitons of the dNLS equation (\ref{dnls-introduction}) are not nearly
compactly supported due to the fact that the dNLS equation
(\ref{dnls-introduction}) has no small parameter. In agreement with
this picture, the continuation arguments in \cite{MA94,PelSak} are not
valid in the small-amplitude approximation, when the breather period
$T$ is defined near the linear limit $2 \pi$, because the inverse
linearized operators become unbounded in the linear oscillator limit
as $T \to 2 \pi$.

By Theorem \ref{theorem-justification}, discrete solitons of the dNLS
equation (\ref{dnls-introduction}) are continued as approximate
breather solutions of the dKG equation (\ref{KGlattice}), which are
only periodic solutions up to the time scale
$\mathcal{O}(\epsilon^{-1})$.  However, this approximation can be
extended to all times, by considering time-periodic solutions of the
dKG equation (\ref{KGlattice}) with small $\epsilon$, using Fourier
series in time, and eliminating all but the first Fourier harmonic by
a Lyapunov--Schmidt reduction procedure. Then, the components of the
first Fourier component satisfies a stationary dNLS-type equation,
where the dNLS equation (\ref{dnls-introduction}) is the leading
equation. In this way, similarly to the work \cite{PSM}, one can
justify continuation of discrete solitons of the dNLS equation
(\ref{dnls-introduction}) as approximate solutions of the true
breathers in the dKG equation (\ref{KGlattice}).

Within the same scheme of Lyapunov--Schmidt decomposition, another
equivalent route to prove the existence of breather for the dKG
equation (\ref{KGlattice}) is obtained by means of
Theorem~\ref{prop.gen}. Indeed, the discrete solitons of the dNLS
equation (\ref{dnls-introduction}) can be characterized as constrained
critical points of energy, which are continued, under non-degeneracy
conditions, to critical points of the true energy of the dKG equation
(\ref{KGlattice}).\\

{\bf 2. Spectral stability of breathers.} Spectral stability of
breathers in the dKG equation (\ref{KGlattice}) can be related to the
spectral stability of solitons in the dNLS equation
(\ref{dnls-introduction}).  By Theorem \ref{theorem-justification}, we
are not able to relate stable or unstable eigenvalues of the dNLS solitons
with the Floquet multipliers of the dKG breathers, because the error
term also grows exponentially at the time scale
$\mathcal{O}(\epsilon^{-1})$ (the same problem is discussed in
\cite{DumaPel} in the context of stability of the travelling waves in
FPU lattices).  However, by Theorem
\ref{theorem-justification-extended} obtained on the extended time
scale $\mathcal{O}(\epsilon^{-1} |\log(\epsilon)|)$, we can conclude
that all unstable eigenvalues of the dNLS solitons persist as unstable
Floquet multipliers of the dKG breathers within the
$\mathcal{O}(\epsilon)$ distance from the unit circle.

If unstable eigenvalues of the dNLS solitons do not exist, we only
obtain approximate spectral stability of the dKG breathers, because
unstable Floquet multipliers of the dKG breathers may still exist on
the distance smaller the $\mathcal{O}(\epsilon)$ to the unit circle.
On the other hand, if the spectrally stable dNLS solitons are known to
have internal modes \cite{PelSak11}, then the Floquet multipliers of
the dKG breathers persist on the unit circle by known symmetries of
the Floquet multipliers \cite{PelSak}. \\

{\bf 3. Long time stability of breathers.} By means of the normal form
approach, it is possible to prove the long time stability result for
single-site (fundamental) breather solutions of the dKG equation
(\ref{KGlattice}). Indeed, the variational characterization of the
existence problem for such breathers in the normal form essentially
implies an orbital stability in the normal form, which is translated
into a long time stability in the original dKG equation \cite{PP15}.

In the case of multi-site dNLS solitons, nonlinear instability is
induced by isolated internal modes of negative Krein signature, which
are coupled with the continuous spectrum by nonlinearity \cite{KPS15}.
By using the extended time scale $\mathcal{O}(\epsilon^{-1}
|\log(\epsilon)|)$ of Theorem \ref{theorem-justification-extended}, we
can predict persistence of this instability for small-amplitude dKG
breathers.  This was recently confirmed for multi-site dKG
breathers in \cite{CKP15}.

Also quasi-periodic localized solutions were constructed for the dNLS
equation in the situation, when the internal mode of the dNLS soliton
occurs on the other side of the spectral band of the continuous
spectrum \cite{Cuccagna,Maeda}. These solutions correspond
approximately to quasi-periodic dKG breathers. It
is still an open question to consider true quasi-periodic breather
solutions of the dKG equation (\ref{KGlattice}).\\

{\bf 4. Asymptotic stability of breathers.} Asymptotic stability of
dNLS solitons supported at a single-site linear potential was
considered in \cite{CT09,KPS09,MP12} within the dNLS equation with the
seventh-order power nonlinearity. Again, these results apply to the
dKG equation only approximately on the extended time scale
$\mathcal{O}(\epsilon^{-1} |\log(\epsilon)|)$ by Theorem
\ref{theorem-justification-extended}.  However, by using canonical
transformations and dispersive decay estimates, asymptotic stability
of single-site breathers was proved recently in \cite{Bambusi2}, also
in the case of the seventh-order power nonlinearity.

The multi-site dKG breathers with isolated internal modes of positive
Krein signature, which are coupled to the spectral band of the
continuous spectrum due to nonlinearity, are also expected to remain
nonlinearly (and, perhaps, asymptotically) stable \cite{CKP15}.  It
remains however an open problem to study this problem directly for the dKG equation (\ref{KGlattice}),
without rescuing to the approximation result of Theorem \ref{theorem-justification-extended}.

\vspace{0.25cm}

{\bf Acknowledgements:} The work of D.P. is supported by the
Ministry of Education and Science of Russian Federation (the base part
of the state task No. 2014/133, project No. 2839). The work of T.P. and S.P. is partially
supported by the MIUR-PRIN program under the grant 2010 JJ4KPA
(``Teorie geometriche e analitiche dei sistemi Hamiltoniani in
dimensioni finite e infinite'').

\end{document}